\documentclass[12pt]{amsart}

  \usepackage{caption}
  \usepackage[margin=3cm, marginparwidth=2.6cm]{geometry}

  \usepackage{amsmath,amssymb,amsthm,amsfonts,bbding,stmaryrd,dsfont,wasysym,pifont,xspace,xcolor}
  \usepackage{parskip,fancyhdr,url}

  \usepackage{stackrel}

  \usepackage{epstopdf,yfonts}
  \usepackage{graphicx,enumerate,epsfig}
  \usepackage{mathrsfs}
  \usepackage{enumitem}
  \usepackage[normalem]{ulem}

\newtheorem{thm}{Theorem}[section]

\newtheorem{ques}[thm]{Question}

\newtheorem{cor}[thm]{Corollary}
\newtheorem{lem}[thm]{Lemma}

\theoremstyle{definition}
\newtheorem{defn}[thm]{Definition}

\newtheorem{remark}[thm]{Remark}

\usepackage{fancybox}

\usepackage{ulem}

\renewcommand{\emph}[1]{{\it #1}}

\usepackage{color}

\title{Degree-regular triangulations of surfaces}

\author{Basudeb Datta and Subhojoy Gupta}
\address{Department of Mathematics, Indian Institute of Science, Bangalore 560012, India.}

\thanks{This work was supported by the Infosys Foundation.}

\begin{document}
\setcounter{tocdepth}{4}
\maketitle

\begin{abstract}
A degree-regular triangulation is one in which each vertex has identical degree. 
Our main result is that any such  triangulation of a (possibly non-compact) surface  $S$ is geometric, that is, it is combinatorially equivalent to a geodesic triangulation with respect to a constant curvature metric on $S$, and we list the possibilities. A key ingredient of the proof is to show that any two $d$-regular triangulations of the plane for $d> 6 $ are combinatorially equivalent. The proof of this uniqueness result, which is of independent interest, is based on an inductive argument involving some combinatorial topology. 
\end{abstract}

\section{Introduction}

A \textit{triangulation} of a surface $S$ is an embedded graph whose complementary regions are faces bordered by exactly three edges (see \S2 for a more precise definition). Such a triangulation is $d$-regular if the valency of each vertex is exactly $d$, and  \textit{geometric} if the triangulation are geodesic segments with respect to a Riemannian metric of constant curvature on $S$, such that each face is an equilateral triangle. 
Moreover, two triangulations $\mathcal{G}_1$ and $\mathcal{G}_2$  of $S$ are \textit{combinatorially equivalent} if there is a homeomorphism $h:S\to S$ that induces a graph isomorphism from $\mathcal{G}_1$ to $\mathcal{G}_2$. 

The \textit{existence} of such degree-regular geometric  triangulations on the simply-connected model spaces ($\mathbb{S}^2, \mathbb{R}^2$ and $\mathbb{H}^2$) is well-known.
Indeed, a $d$-regular geometric triangulation of the hyperbolic plane $\mathbb{H}^2$ is generated by the Schwarz triangle group $\Delta(3,d)$  of reflections across the three sides of a hyperbolic triangle (see \S3.2 for a discussion).

In this note we prove:
\begin{thm}\label{thm0}  Any $d$-regular triangulation of a surface $S$ is combinatorially equivalent to a geometric triangulation, and one of the following three possibilities hold:

\begin{itemize}
\item $d <6$,   $S \cong S^2$ and the triangulation corresponds to the boundary of the tetrahedron, octahedron or icosahedron,
or  $S \cong \mathbb{R}\mathrm{P}^2$ and the triangulation is a two-fold quotient of the boundary of the icosahedron,

\item $d =6$,  the universal cover  $\tilde{S}= \mathbb{E}^2$ and the lift of the triangulation is the standard hexagonal  triangulation, 

\item $d >6$,  $\tilde{S}= \mathbb{E}^2$ and the lift of the triangulation is the geodesic triangulation generated by the Schwarzian triangle group $\Delta(3,d)$.
\end{itemize}

\end{thm}

In the case when $S$ is compact, this was proved in \cite{Kulkarnietal}; indeed, our work also gives a new proof of their result.

A key ingredient in our proof of Theorem \ref{thm0} is the following uniqueness result for triangulations of the plane:

\begin{thm}\label{thm1} For each integer $d> 6$, any two $d$-regular triangulations of $\mathbb{R}^2$ are combinatorially equivalent.
\end{thm}

In previous work \cite{Datta1} (see also \cite{Datta1b}),  it was shown that the above result is true for $d=6$ as well; in this case the triangulation is combinatorially equivalent to the familiar hexagonal  triangulation of the Euclidean plane $\mathbb{E}^2$.

This paper develops a different technique that works for the case of a degree-regular triangulation of the hyperbolic plane $\mathbb{H}^2$, where the number of vertices in a compact subcomplex grows exponentially with the radius  (see Lemma \ref{vcount}). Our proof involves an inductive procedure to construct an exhaustion of the plane by compact triangulated regions; we start with the link of a single vertex, and at each step, add all the links of the boundary vertices. A crucial part of the argument is to ensure that at each step, the region we obtain is a triangulated disk that is determined uniquely by the preceding step; this needs a topological argument involving the Euler characteristic.

A uniqu{eness result akin to Theorem \ref{thm1} is the uniqueness  of  \textit{Archimedean tilings},  which was recently proved in \cite{Datta2} and forms the inspiration for the present work.



Note that a  \textit{tiling} (or {\em map}) of a  surface is a more general notion when the embedded graph has complementary faces that are (topological) polygons. The \textit{type} of a vertex in a tiling is the cyclic collection of integers that determines the number of sides of each polygon surrounding it.
If all the vertices in a tiling are of same type then the tiling is called a {\em semi-equivelar map} or a {\em uniform tiling}.
An {\em Archimedean} tiling of the plane $\mathbb{R}^2$ is a semi-equivelar map of $\mathbb{R}^2$ by regular Euclidean polygons; it is a classical fact that there are exactly $11$ such tilings of $\mathbb{R}^2$ (see also \cite{G-S}).




The techniques of this paper can be extended to show the uniqueness of \textit{equivelar maps}, which are uniform tilings where each vertex is of type $[p^q]$, that is, each tile is a $p$-gon and there are exactly $q$ such polygons around each vertex. This would imply that any equivelar map of a surface is geometric. 
Once again, this was shown by \cite{Kulkarnietal}, but only in the case when $S$ is compact.

However the following question is, to the best of our knowledge,  still open:

\begin{ques} Are any two uniform tilings of $\mathbb{R}^2$ with the same vertex type  combinatorially equivalent?
\end{ques}

We hope to address this in a subsequent article.



\section{Preliminaries}

We begin by setting up some definitions and notation that we shall use throughout the paper.

A standard reference for
basic terminology on graphs is \cite{BM}. For a graph $G$, $V(G)$ and
$E(G)$ will denote its vertex-set and edge-set respectively, 

\begin{defn}[Triangulation] A triangulation of a surface $S$ is an embedded graph $G$ such that
\begin{itemize}
\item all complementary regions (\textit{faces}) are (topological) triangles,
\item each edge in $E(G)$  belongs to (the closure of) exactly two triangles,
\item each vertex  in $V(G)$ belongs to at least $3$ triangles, arranged in a cyclic order around it.
\end{itemize}
\end{defn}

Further, for $v\in
V(G)$, $d_G(v)$ denotes the {\em degree} of $v$ in $G$. 
A triangulation is \textit{$d$-regular} for $d\geq 3$ if moreover, each vertex  has valency (degree) $d$.

A {\em path} with edges $u_1u_2, \dots, u_{n-1}u_n$ is denoted by $u_1\mbox{-}u_2\mbox{-} \cdots \mbox{-}u_n$. For $n\geq 3$, an \textit{$n$-cycle} with edges $u_1u_2, \dots, u_{n-1}u_n, u_nu_1$ is denoted by $u_1\mbox{-}u_2\mbox{-} \cdots \mbox{-}u_n\mbox{-}u_1$ and also by
$C_n(u_1, u_2, \dots, u_n)$.

If $G$ is a graph and $u$ is a vertex not in $G$ then the {\em join} $u\ast G$ is the simplicial complex whose vertex set is $V(G)\sqcup\{u\}$, edge-set is $E(G) \sqcup \{uv \, : \, v\in V(G)\}$ and the set of 2-simplices is $\{uvw \, : \, vw\in E(G)\}$.

By a {\em triangulated surface}, we mean a simplicial complex whose underlying topological space is a surface (possibly with boundary), such that its $1$-skeleton is a triangulation as defined above. 
Since a closed surface can not be a proper subset of a surface, it follows that a \textit{closed} triangulated surface can't be a proper subcomplex of a triangulated surface. We identify a triangulated surface by the set of 2-simplices in the complex.

If $u$ is a vertex in a 2-dimensional simplicial complex $X$ then the \textit{link} of $u$ in $X$, denoted by ${\rm lk}_X(u)$ (or simply ${\rm lk}(u)$), is the largest subgraph $G$ of $X$ such that $u\not\in G$ and $u\ast G \subseteq X$. The subcomplex $u\ast {\rm lk}_X(u)$ is called the {\em star} of $u$ in $X$ and is denoted by ${\rm st}_X(u)$ or simply by ${\rm st}(u)$. 

Clearly, the star of a vertex in a triangulated surface is a triangulated 2-disc. A 2-dimensional simplicial complex is a triangulated surface if and only if the links of vertices in $X$ are paths and cycles. Moreover, for a triangulated surface $X$, ${\rm lk}_X(u)$ is a path if and only if $u$ is a boundary vertex of $X$. If not mentioned otherwise, by a triangulated surface we mean a triangulated surface without boundary.


\section{Geometric triangulations of surfaces}

In this section we discuss constructions of $d$-regular triangulated surfaces where the triangulations are geometric. 

\subsection{Case of degree $d< 6 $}

The following result lists all possibilities of a $d$-regular triangulated surface when the degree $d<6$.

\begin{lem}\label{lem1}
Let $X$ be a triangulated surface without boundary.

\begin{enumerate}
\item[{\rm (i)}] If $X$ is $3$-regular, then $X$ is the boundary of a tetrahedron.
\item[{\rm (ii)}] If $X$ is $4$-regular, then $X$ is the boundary of an octahedron.
\item[{\rm (iii)}] If $X$ is $5$-regular, then $X$ is either the boundary of an icosahedron, or $\mathbb{R}P^2_6$  (the minimal $6$-vertex triangulation of the real projective plane).
\end{enumerate}

\end{lem}

\begin{remark}
{\rm Note that in the following proof of Lemma \ref{lem1} we do not assume that  $X$ is finite.}
\end{remark}

\setlength{\unitlength}{1.75mm}

\begin{picture}(83,49.5)(-15,-6.5)

{\boldmath{

\thicklines

\put(0,0){\line(1,0){60}} \put(0,0){\line(3,4){30}}
\put(0,0){\line(5,1){30}} \put(0,0){\line(5,3){20}}
\put(60,0){\line(-3,4){30}} \put(60,0){\line(-5,3){20}}
\put(60,0){\line(-1,1){20}} \put(20,20){\line(1,2){10}}
\put(30,24){\line(0,1){16}}

\put(20,12){\line(5,-3){10}} \put(20,12){\line(5,6){5}}
\put(20,12){\line(0,1){8}}

\put(30,12){\line(1,0){10}} \put(30,12){\line(-5,6){5}}
\put(30,12){\line(5,6){5}}

\put(30,6){\line(5,3){10}} \put(30,6){\line(0,1){6}}

\put(40,12){\line(0,1){8}}

\put(20,20){\line(5,-2){5}} \put(20,20){\line(5,2){10}}
\put(40,20){\line(-5,-2){5}} \put(40,20){\line(-5,2){10}}

\put(25,18){\line(1,0){10}} \put(35,18){\line(-5,6){5}}

}}

\thinlines

\put(-15,-6.5){\line(1,0){80}} \put(-15,43){\line(1,0){80}}

\put(0,0){\line(1,1){2.1}} \put(2.55,2.55){\line(1,1){2.1}}
\put(5.1,5.1){\line(1,1){2.1}} \put(7.65,7.65){\line(1,1){2.1}}
\put(10.2,10.2){\line(1,1){2.1}}
\put(12.75,12.75){\line(1,1){2.1}}
\put(15.3,15.3){\line(1,1){2.1}}
\put(17.85,17.95){\line(1,1){2.15}}

\put(30,6){\line(5,-1){3}} \put(33.5,5.3){\line(5,-1){3}}
\put(37,4.6){\line(5,-1){3}} \put(40.5,3.9){\line(5,-1){3}}
\put(44,3.2){\line(5,-1){3}} \put(47.5,2.5){\line(5,-1){3}}
\put(51,1.8){\line(5,-1){3}} \put(54.5,1.1){\line(5,-1){3}}
\put(60,0){\line(-5,1){2.1}}

\put(40,20){\line(-1,2){1.5}} \put(38.3,23.4){\line(-1,2){1.5}}
\put(36.6,26.8){\line(-1,2){1.5}}
\put(34.9,30.2){\line(-1,2){1.5}}
\put(33.2,33.6){\line(-1,2){1.5}} \put(31.5,37){\line(-1,2){1.5}}

\put(20,12){\line(1,0){2.25}} \put(22.75,12){\line(1,0){2}}
\put(25.25,12){\line(1,0){2}} \put(27.75,12){\line(1,0){2.25}}

\put(40,12){\line(-5,6){1.68}} \put(38,14.4){\line(-5,6){1.68}}
\put(36,16.8){\line(-5,6){1.68}}

\put(30,24){\line(-5,-6){1.68}} \put(28,21.6){\line(-5,-6){1.68}}
\put(26,19.2){\line(-5,-6){1.68}}

\put(-3,1){\mbox{$4^{\prime}$}} \put(60.5,1){\mbox{$3^{\prime}$}}
\put(32,39.5){\mbox{$0^{\prime}$}} \put(31,24.5){\mbox{$1^{\prime}$}}
\put(31,10){\mbox{$0$}} \put(29,4){\mbox{$1$}}
\put(41,12.5){\mbox{$5$}} \put(41,20){\mbox{$2^{\prime}$}}
\put(16.5,12.5){\mbox{$2$}} \put(16.5,20){\mbox{$5^{\prime}$}}
\put(27,18.75){\mbox{$3$}} \put(34.4,19.7){\mbox{$4$}}

\put(30,-3){\mbox{$\partial \, {\mathcal I}$}}

\put(20,-5){\mbox{\bf Figure 0}}

\thicklines

\put(-10,28){\line(5,-3){10}} \put(-10,28){\line(5,6){5}}
\put(-10,28){\line(0,1){8}}

\put(0,28){\line(1,0){10}} \put(0,28){\line(-5,6){5}}
\put(0,28){\line(5,6){5}}

\put(0,22){\line(5,3){10}} \put(0,22){\line(0,1){6}}

\put(10,28){\line(0,1){8}}

\put(-10,36){\line(5,-2){5}} \put(-10,36){\line(5,2){10}}
\put(10,36){\line(-5,-2){5}} \put(10,36){\line(-5,2){10}}

\put(-5,34){\line(1,0){10}} \put(5,34){\line(-5,6){5}}


\put(-10,28){\line(1,0){10}} \put(10,28){\line(-5,6){5}}
\put(-5,34){\line(5,6){5}}

\put(1,40.5){\mbox{$1$}} \put(1,26){\mbox{$0$}}
\put(-3,21){\mbox{$1$}} \put(11,28.5){\mbox{$5$}}
\put(11,36){\mbox{$2$}} \put(-12.5,28.5){\mbox{$2$}}
\put(-12.5,36){\mbox{$5$}} \put(-3,34.75){\mbox{$3$}}
\put(4.4,35.7){\mbox{$4$}}

\put(2,19){\mbox{$\mathbb{R P}^{\,2}_{6}$}}




\end{picture}


\begin{proof}
\textit{Part} (i). Let $a$ be a vertex whose link is $C_3(b, c, d)$. Then $c\mbox{-}a\mbox{-}d \subseteq {\rm lk}_X(b)$. Since, $\deg(b)=3$, it follows that ${\rm lk}_X(b)= C_3(c, a, d)$. Thus, $bcd$ is a simplex and hence the boundary $\partial \Delta$ of the tetrahedron $\Delta = abcd$ is a subcomplex of $X$. Since, $\partial \Delta$ is a closed triangulated surface, it follows that $X= \partial \Delta$.

\noindent \textit{Part} (ii).  Let $a$ be a vertex whose link is $C_4(b, c, d, e)$. Then $c\mbox{-}a\mbox{-}e \subseteq {\rm lk}_X(b)$. Since, $\deg(b)=4$, it follows that ${\rm lk}_X(b)= C_3(c, a, e, f)$ for some vertex $f$. If $d=f$ then $C_3(d, a, b) \subseteq {\rm lk}_X(c)$. This is not possible. So, $f$ is a new vertex. Then $d\mbox{-}a\mbox{-}b\mbox{-}f \subseteq {\rm lk}_X(c)$ and hence ${\rm lk}_X(c)= C_4(d, a, b, f)$. Similarly, ${\rm lk}_X(e)= C_4(d, a, b, f)$. Then $abc, acd, ade, aeb$, $bcf, cdf, def, ebf$ are simplices in $X$. Since these eight 2-simplices form the boundary of an octahedron $O$, it follows $X= \partial O$.

\textit{Part} (iii). Let $0$ be a vertex and ${\rm lk}_X(0) = C_5(1,2,3,4,5)$.
Let $34x\neq 034$ be another 2-simplex containing the edge $34$.
If $x=5$ then $C_3(5,0,3) \subseteq {\rm lk}(4)$, a contradiction. So, $x\neq 5$. Similarly, $x\neq 2$. So, $x=1$ or a new vertex.

First assume that $x=1$. Then the path $5\mbox{-}0\mbox{-}2$ and the edge $34$ are part of ${\rm lk}_X(1)$. Since $123$ is not a simplex (otherwise $C_3(3, 0, 1)$ would be a part of ${\rm lk}_X(2)$), it follows that ${\rm lk}_X(1)= c_5(5, 0, 2, 4, 3)$. Then $5\mbox{-}0\mbox{-}3\mbox{-}1\mbox{-}2 \subseteq {\rm lk}_X(4)$ and hence ${\rm lk}_X(4)= C_5(5, 0, 3, 1, 2)$. Similarly, ${\rm lk}_X(3)= C_5(2, 0, 4, 1, 5)$.
Then the closed triangulated surface $\mathbb{RP}^2_6$ (see Fig. 1) is a subcomplex of $X$ and hence $X= \mathbb{RP}^2_6$.

If $235$, $341$, $452$ or $513$ is a simplex then by the same argument as above it follows that $X=\mathbb{RP}^2_6$. So, assume that $124$, $235$, $341$, $452$, $513$ are not simplices.

Therefore, let $341^{\prime}$ be a simplex where $1^{\prime}$ is a new vertex.
Then ${\rm lk}_X(4)$ is of the form $C_5(5, 0, 3, 1^{\prime}, y)$ for some vertex $y$. 
If $y=1$ then $C_3(1, 0, 4) \subseteq {\rm lk}_X(5)$. Thus, $y\neq 1$. Since $245$ is not a simplex, $y\neq 2$. So, $y$ is a new vertex, say $y=2^{\prime}$. Thus ${\rm lk}_X(4) = C_5(5, 0, 3, 1^{\prime}, 2^{\prime})$.
This implies that ${\rm lk}_X(5) = C_5(1, 0, 4, 2^{\prime}, z)$ for some vertex $z$. If $z = 1^{\prime}$ then $C_3(5, 4, 1^{\prime})\subseteq {\rm lk}_X(2^{\prime})$, a contradiction.
So, ${\rm lk}_X(5) = C_5(1, 0, 4, 2^{\prime}, 3^{\prime})$ for some new vertex $3^{\prime}$.
Then ${\rm lk}_X(1)$ would be of the form $C_5(2, 0, 5, 3^{\prime}, w)$ for some vertex $w$. By similar argument, $w\neq 1^{\prime}$. If $w=2^{\prime}$ then $2^{\prime}4, 2^{\prime}5, 2^{\prime}1^{\prime}, 2^{\prime}3^{\prime}, 2^{\prime}1, 2^{\prime}2$ are edges and hence $\deg(2^{\prime}) > 5$, a contradiction. So, $w$ a new vertex. Thus,
${\rm lk}_X(1) = C_5(2, 0, 5, 3^{\prime}, 4^{\prime})$ for some new vertex $4^{\prime}$. By similar argument, ${\rm lk}_X(2) = C_5(3, 0, 1, 4^{\prime}, 5^{\prime})$ for some new vertex $5^{\prime}$.
Then $1^{\prime}\mbox{-}4\mbox{-}0\mbox{-}2\mbox{-}5^{\prime} \subseteq {\rm lk}_X(3)$ and hence ${\rm lk}_X(2)=C_5(3, 0, 1, 4^{\prime}, 5^{\prime})$ (see Fig. 1).

Now,  $5^{\prime}\mbox{-}3\mbox{-}4\mbox{-}2^{\prime} \subseteq {\rm lk}_X(1^{\prime})$. Thus, ${\rm lk}(1^{\prime}) = C_5(5^{\prime}, 3, 4, 2^{\prime}, a)$ for some vertex $a$. If $a=3^{\prime}$ then $C_4(3^{\prime}, 1^{\prime}, 4, 5) \subseteq {\rm lk}_X(2^{\prime})$, a contradiction. So, $a\neq 3^{\prime}$. Similarly, $a\neq 4^{\prime}$. Thus, $a = 0^{\prime}$ for some new vertex $0^{\prime}$. Then $0^{\prime}\mbox{-}1^{\prime}\mbox{-}4\mbox{-}5 \mbox{-}3^{\prime} \subseteq {\rm lk}_X(2^{\prime})$ and hence ${\rm lk}_X(2^{\prime})= C_5(0^{\prime}, 1^{\prime}, 4, 5, 3^{\prime})$. Thus, $0^{\prime}2^{\prime}3^{\prime}$ is a simplex. Similarly, from the links of $5^{\prime}$ and $0^{\prime}$, we get the simplices $0^{\prime}4^{\prime}5^{\prime}$ and $0^{\prime}3^{\prime}4^{\prime}$ respectively.
Then the closed triangulated surface $\partial\,{\mathcal I}$ (boundary of an icosahedron ${\mathcal I}$) is a subcomplex of $X$. This implies $X= \partial\,{\mathcal I}$. This completes the proof.
 \end{proof}

 \subsection{Case of degree $d\geq 6$}

\begin{lem}\label{lem2} A triangulated surface with a $6$-regular triangulation admits a Euclidean (zero curvature) metric, and a $d$-regular triangulation for $d>6$ admits a hyperbolic (constant negative curvature) metric. 
\end{lem}

\begin{proof}
For $d=6$, equip each triangle with a flat (Euclidean) metric in which each side is of the same length, and each interior  angle is $\pi/3$. Since there are 6 triangles around each vertex of the triangulation, one obtains a smooth flat metric on the entire triangulated surface. (Note that the metric matches on the sides since all sides are of the same length). 

Similarly, for $d>6$, we equip each triangle with a hyperbolic metric such that each side is of equal length and  each interior angle is $2\pi/d$. Once again, since there are $d$ triangles around each vertex, the total angle around each vertex is $2\pi$ and one obtains a smooth metric of constant negative curvature on the surface.  \end{proof}

 \begin{cor} A  triangulated surface with a $d$-regular triangulation for $d\geq 6$ must have universal cover $\mathbb{E}^2$ or $\mathbb{H}^2$, both homeomorphic to the plane $\mathbb{R}^2$.
 \end{cor}

 Thus, for a surface $S$ with universal cover $\mathbb{R}^2$,  the \textit{existence} of geometric degree-regular triangulations on $S$  when  $d\geq 6$ reduces to the existence of such triangulations on the plane, by passing to the universal cover:

A $6$-regular triangulation of $\mathbb{E}^2$ is the usual hexagonal tiling, which is geometric with respect to the standard Euclidean metric. 
A $d$-regular triangulation of $\mathbb{H}^2$ for $d>6$ can be constructed by taking a hyperbolic geodesic triangle $T_0$ with interior angles $\frac{\pi}{2}, \frac{\pi}{3}$ and $\frac{\pi}{d}$, and considering the Schwarz triangle group
\begin{equation*}
\Delta(3,d) = \langle a, b, c  \text{ } \vert\text{ }a^3= b^d = c^2 =1 \rangle 
\end{equation*}
which is generated by reflections on the three sides.

This is a discrete (Fuchsian) group and the orbit of $T_0$ under its action is a geometric triangulation of $\mathbb{H}^2$ which is a first barycentric subdivision of a $d$-regular triangulation. (See Figure \ref{reg-tiling}.) In fact, the symmetry group the resulting triangulation is $\Delta(3,d)$; for a discussion, historical references, and an extension of the construction to include other uniform tilings with symmetry group $\Delta(p,q)$, refer to Section 8 of \cite{Kulkarnietal}.



\begin{figure}
  \centering
  \includegraphics[scale=.45]{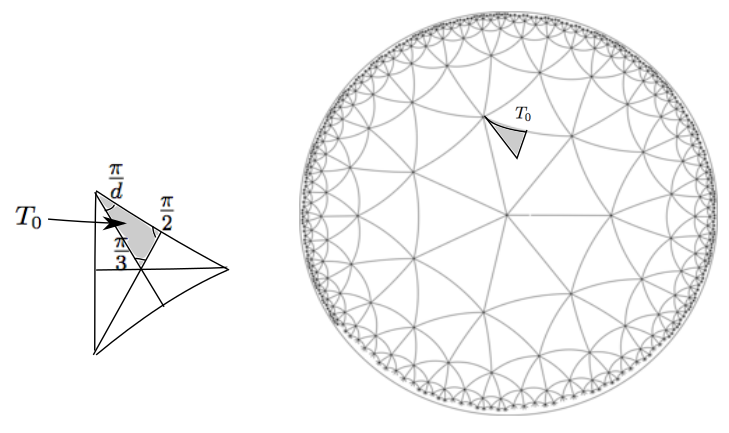}\\
  \caption{A $7$-regular triangulation of the hyperbolic plane is associated to the $(2,3,7)$-triangle group that reflects along the sides of a geodesic triangle $T_0$. }
  \label{reg-tiling}
\end{figure}

\section{Proof of Theorem \ref{thm1}}

 Throughout this section, let $X$ be some $d$-regular triangulation of $\mathbb{R}^2$, where $d\geq 7$.

 We shall prove the uniqueness result by an inductive argument; namely, we show that any such triangulation has a compact exhaustion  $$X_0 = \mathrm{st}(v_0)  \subset X_1 \subset X_2 \subset \cdots \subset  X_k \subset X_{k+1} \subset  \cdots$$

 by triangulated disks (see \S3.2) such that $X_{k+1}$ is determined uniquely by $X_k$ (see \S 4.4).



\subsection{Key lemma}

The following observation will be used repeatedly.

\begin{lem}\label{lemma1} There does not exist a triangulated disk $D$ such that\\
(i) $D \subset X$, (ii) $\# \text{ vertices of }D \geq 4$, and (iii) degrees of all boundary vertices, except possibly three, are $\geq d-2$.
\end{lem}

\begin{proof}

  We shall obtain a contradiction to show such a triangulated disk $R$ cannot exist,  by an Euler characteristic argument as follows:\\

    Let  $$m_2 = \# \text{ boundary  vertices of } R \text{ of degree }  d-2,$$
        $$m_1 = \# \text{ boundary vertices of } R \text{ of degree }  d-1, $$
    $$m_0 = \# \text{ vertices of } R \text{ of degree }  d =  \# \text{ internal vertices of } R .$$


  Also, let $d_x$, $d_v$ and $d_w$ be the degree of the three exceptional  vertices $x$, $v$ and $w$ respectively.

  Note that by assumption (ii), we have $m\geq 1$.

  Let $f_0,f_1$ and $f_2$ be the number of vertices, edges and triangles of $R$ respectively.

Then we have:
  \begin{equation}
  f_0 =  3 + m_0+ m_1+m_2 ,
  \end{equation}

  where the $3$  corresponds to the vertices $x,v$ and $w$.

 We get a count of the number of edges by considering the degrees of the vertices:
 \begin{equation}
  f_1 =  \frac{d_v + d_w + d_x +  d\cdot m_0 + (d-1) \cdot m_1 + (d-2) \cdot m_2}{2}
    \end{equation}

    and we get the count
     \begin{equation}
  f_2 =  \frac{(d_v-1) + (d_w-1) + (d_x-1) +  d\cdot m_0 + (d-2) \cdot m_1 + (d-3) \cdot m_2}{3}
    \end{equation}

Thus we calculate:
\begin{align*}
f_0-f_1 + f_2 =  \left(\frac{12- (d_x + d_v + d_w)}{6}\right)  + m_0\left(\frac{6-d}{6}\right) + m_1\left( \frac{5-d}{6}\right)  +  m_2\left(\frac{6-d}{6}\right)
\end{align*}
and so
\begin{align*}
f_0-f_1 + f_2 <  \left(\frac{12- (d_x + d_v + d_w)}{6}\right)
\end{align*}
since the remaining terms are negative because $d\geq 7$.

Since $x,v,w$ are boundary vertices, each has at least two adjacent vertices on $\partial D$, and we have
\begin{equation*}
d_x + d_v + d_w \geq 6.
\end{equation*}

Using this, we obtain
\begin{equation*}
\chi(R) = f_0 - f_1 + f_2  <  \left(\frac{12- 6}{6}\right) =  1
\end{equation*}
which contradicts the fact that $R$ is a disk and hence $\chi(R) =1$.
\end{proof}

\medskip

As special cases, we obtain:\\

\begin{cor}\label{cor1}

(i) The subcomplex of $X$ shown in Figure 2 does not exist:

\begin{figure}[h]
  \centering
  \includegraphics[scale=.35]{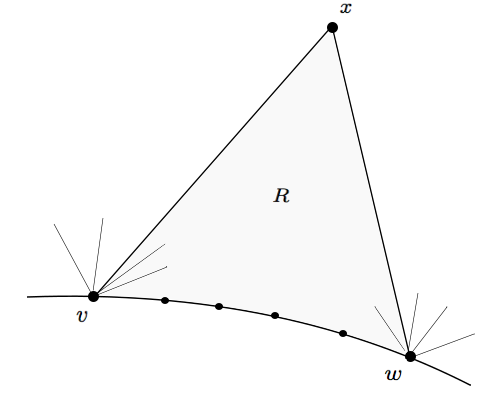}\\
  \caption{  }
  \label{case1}
\end{figure}

where all boundary vertices of the triangulated disks  (except $x$, $v$, $w$) have degrees $d-1$ or $d-2$. \\

(ii) The  subcomplex of $X$  shown in Figure 3 does not exist:

\begin{figure}[h]
  \centering
  \includegraphics[scale=.35]{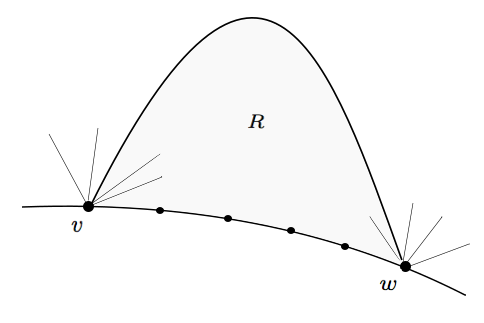}\\
  \caption{  }
  \label{case2}
\end{figure}

where all boundary vertices of the triangulated disks  (except $v$, $w$) have degrees $d-1$ or $d-2$. \\

\end{cor}

 \vspace{.5in}

\begin{cor}\label{cor2}
If $uv$ is an edge then the number of vertices adjacent to both $u$ and $v$ is exactly $2$.
\end{cor}
\begin{proof}
Since $uv$ belongs to two $2$-simplices, $u$ and $v$ have at least two common vertices. If possible let $w$ be another common vertex. Then $uvw$ is a non-simplex and hence $u-v-w$ is in the boundary of a disk where all the other vertices are interior vertices which form a non-empty set. This is not possible by Lemma \ref{lemma1}.
\end{proof}

\subsection{Layered structure}

We begin by defining a compact exhaustion of $X$ by triangulated disks $\{X_k\}_{k\geq 1}$.

Fix a vertex $v_0$ in $X$.

Let
\begin{equation*}
V_0 = \{v_0\},
\end{equation*}
\begin{equation*}
V_1 = \text{ the set of }d\text{ neighbours of } V_0, \text{ and} 
\end{equation*}
\begin{equation*}
V_k = \{u \vert u \notin \bigcup\limits_{j=0}^{k-1} V_j, u \text{ has a neighbour in } V_{k-1}\}.
\end{equation*}

Let  
\begin{equation*}
X_0 = {\rm st}(v_0), \text{ and }
\end{equation*}
\begin{equation}\label{xk}
 X_k = X_{k-1} \cup \left(\bigcup\limits_{v\in V_{k-1}} \mathrm{st}(v)\right)
 \end{equation}
 for each $k\geq1 $.

Finally, let $L_k = \partial X_k$.\\

\begin{lem}\label{lemma2}
Let $k\geq1$. Then the following are true:
\begin{itemize}
\item[($a_k$)] $X_k$ is a triangulation of the $2$-disc.
\item[($b_k$)] The vertex set $V(\partial X_k)$ is $V_k$.
\item[($c_k$)] $L_{k-1}$ is the sub-complex induced by $V_{k-1}$, denoted $X[V_{k-1}]$. Equivalently, $X[V_{k-1}]$ is a cycle.
 \item[($d_k$)]  $X_{k-1}$ is the induced subcomplex $X\left[ \bigcup\limits_{j=0}^{k-1}V_j \right]$.
  \item[($e_k$)] Let $u,v\in V_{k-1}$ such that $uv$ is not an edge. Then $u$ and $v$ have no common neighbours outside $X_{k-1}$.
   \item[($f_k$)] Each $u\in V_k$ is adjacent to one or two vertices in $V_{k-1}$.
    \item[($g_k$)] $\text{deg}_{X_k}(u) =3 \text{ or }4$ for $u\in V_k$.
\end{itemize}

\end{lem}

\begin{proof}
We prove the results by induction on $k$.  Clearly, the result is true for $k=1$.
Assume that $k\geq 1$ and the result is true for $1,2,\ldots k$.

\textit{Proof of ($c_{k+1}$)}:  This follows from Corollary \ref{cor1} (ii).\\

\textit{Proof of ($d_{k+1}$)}: This follows from ($d_k$) and ($c_{k+1}$).

\textit{Proof of ($e_{k+1}$)}:  This follows from Corollary \ref{cor1} (i). 

\textit{Proof of ($a_{k+1}$)}:  Since $X_k$ is a triangulated $2$-disc, $\partial X_k$ is a cycle on the vertex set $V_k$ (by ($a_k$) and ($b_k$)). Set $\partial X_k = C_m(v_1,\ldots v_m)$.

Then by ($g_k$) and ($c_{k+1}$), we have $$\mathrm{lk}_{X_k}(v_i) = v_{i-1} \mbox{-} a\mbox{-}b\mbox{-}v_{i+1} \text{ or } v_{i-1} \mbox{-} a \mbox{-} v_{i+1}$$ for some $a,b\in V_{k-1}$. See Figure 4 (a) and (b) for these two cases.

Since $\mathrm{lk}_X(v_i)$ is a $2$-cycle, it follows that $$\mathrm{lk}_X(v_i) = \mathrm{lk}_{X_k}(v_i) \cup P_i $$ where $P_i$ is a path (with $d-2$ or $d-1$ vertices) from $v_{i+1}$ to $v_{i-1}$.

Let
\begin{equation}\label{pathP}
P_i = v_{i+1}\mbox{-}p_i\mbox{-}\cdots \mbox{-}q_i\mbox{-}v_{i-1}
\end{equation}

Then  $$Y_1:= X_k \cup \mathrm{st}(v_1) = X_k \cup (v_1 \ast P_1)$$  and by ($d_{k+1}$) we get $X_k \cap (v_1 \ast P_1) = v_2 \mbox{-}v_1\mbox{-}v_{m}$.Thus, $Y_1$ is a triangulated $2$-disc.

In what follows we shall use this principle of ``extending" a triangulated $2$-disk several times.

Next, let $Y_2 = Y_1 \cup \mathrm{st}(v_2)$.
\begin{figure}
  \centering
  \includegraphics[scale=.4]{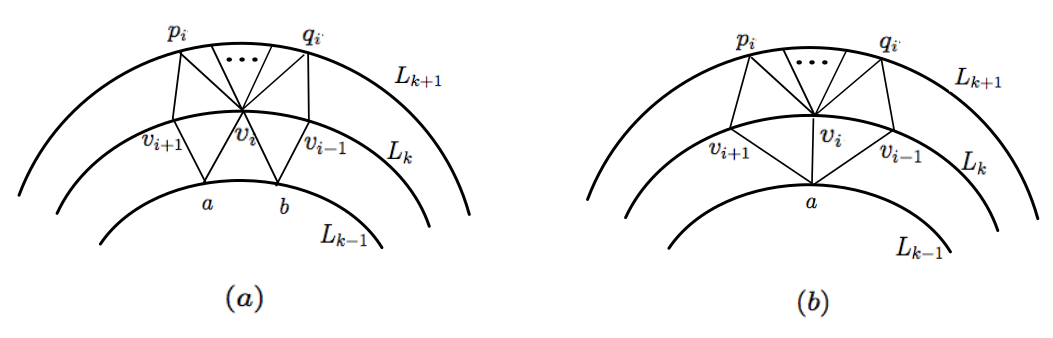}\\
  \caption{  }
\end{figure}
Then  by (\ref{pathP}) we have $$Y_2 = Y_1 \cup (v_2 \ast P_2) = Y_1 \cup (v_2 \ast\mathrm{st}(v_3 \mbox{-} p_2\mbox{-}\cdots \mbox{-} q_2)).$$
Since there exists a unique $2$-simplex of the form $v_1v_2a$ in $X_k$, it follows that there is a unique $2$-simplex containing $v_1v_2$ outside $X_k$. This simplex is $v_1v_2p_1 = v_1v_2q_2$, which implies that $q_2=p_1$ is the same point.  See Figure 5 (a).

By Corollary \ref{cor1} (ii),
$$Y_1 \cap (v_2\ast (v_3 \mbox{-}p_2\mbox{-} \cdots \mbox{-}q_2)  = v_3\mbox{-}v_2\mbox{-}q_2 = v_3\mbox{-}v_2\mbox{-}p_1.$$

 These imply that $Y_2$ is a triangulated $2$-disc.

Now, assume that $Y_i = X_k \cup \left(\bigcup\limits_{j=1}^{i} \mathrm{st}(v_j)\right)$ is a triangulated $2$-disc for $1\leq i\leq m-2$.

Let $Y_{i+1} = Y_i \cup \mathrm{st}(v_{i+1})$. Then as before, $$Y_{i+1} = Y_i \cup (v_{i+1} \ast (v_{i+2} \mbox{-} p_{i+1} \mbox{-}\cdots \mbox{-} q_{i+1}))$$ and $q_{i+1} = p_i$.

By Corollary \ref{cor1}, (i) and  (ii), we obtain
$$Y_i \cap (v_{i+1} \ast (v_{i+2} \mbox{-}p_{i+1} \mbox{-}\cdots \mbox{-} q_{i+1})) = v_{i+2}\mbox{-}v_{i+1}\mbox{-}q_{i+1} = v_{i+2}\mbox{-}y_{i+1}\mbox{-}p_i.$$

These imply that $Y_{i+1}$ is a triangulated disc.

This completes the inductive step, and so  $Y_{m-1} = X_k \cup  \left(\bigcup\limits_{j=1}^{m-1} \mathrm{st}(v_j)\right)$ is a triangulated $2$-disk.

In the final step, note that  $X_{k+1} = Y_{m-1} \cup \mathrm{st}(v_m) = Y_{m-1} \cup (v_m \ast (p_m\mbox{-}\cdots \mbox{-} q_m))$, and $p_m =q_1$, $q_m = p_{m-1}$. See Figure 5 (b).

As before, we can conclude from Corollary \ref{cor1}, (i) and  (ii), that $$Y_{m-1} \cap (v_{m} \ast (p_{m} \mbox{-}\cdots \mbox{-} q_{m})) = p_m\mbox{-}v_m\mbox{-}q_m = q_1\mbox{-}v_m \mbox{-}p_{m-1}.$$

This implies $X_{k+1}$  is a triangulated $2$-disk. This proves $(a_{k+1})$.

\begin{figure}
  \centering
  \includegraphics[scale=.5]{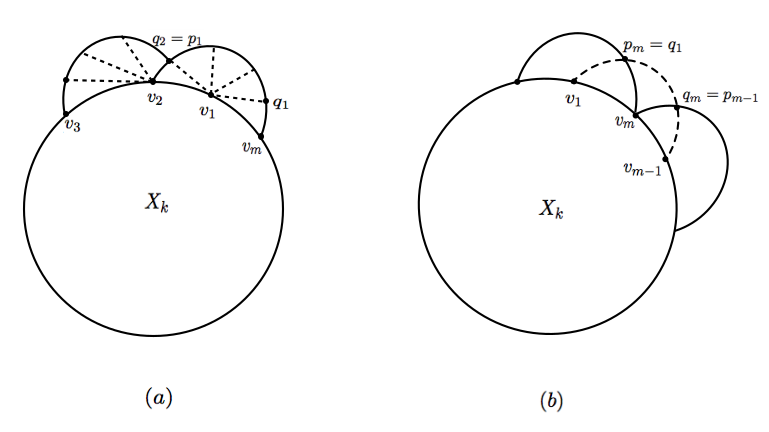}\\
  \caption{ }
\end{figure}

\textit{Proof of ($b_{k+1}$)}:   We shall use some of the notation from the proof of ($a_{k+1}$) above.

If $x\in V(\partial X_{k+1})$ then from the  preceding construction, $x\in p_i\mbox{-}\cdots \mbox{-}q_i \subset P_i \subset \mathrm{lk}_X(v_i)$ for some $v_i\in V_k$ and hence  $x\in V_{k+1}$.

On the other hand, if $y\in V_{k+1}$ then $y\in \mathrm{lk}_X(v_i)$ for some $v_i\in V_k$, and $y\notin X_k$. So, $y\in \mathrm{lk}_X(v_i) \setminus \mathrm{lk}_{X_k}(v_i) = \text{ the path } p_i\mbox{-}\cdots\mbox{-}q_i= P_i^\prime \subset P_i$. Since the path $p_i\mbox{-}\cdots \mbox{-}q_i$ is a part of $\partial X_{k+1}$, by construction, it follows that $y\in V(\partial X_{k+1})$. This proves ($b_{k+1}$).

\textit{Proof of ($f_{k+1}$)}: Since $X_{k+1}$ is a triangulated $2$-disc by ($a_{k+1}$), we know that $\partial X_{k+1}$ is a cycle. So, the paths $P_i^\prime = p_i\mbox{-}\cdots \mbox{-}q_i$ and $P_j^\prime = p_j \mbox{-} \cdots \mbox{-} q_j$ do not intersect internally (possibly only at the end vertices) for $i\neq j$.

Now, let $u\in V_{k+1}$. By ($b_{k+1}$), we have $u \in V(\partial X_{k+1})$ and is in $P_i^\prime $ for some $i\in \{1,2,\ldots ,m\}$. If $u$ is an internal vertex of $P_i^\prime$ then $u$ is adjacent to only $v_i$ of $X_k$.
If $u=p_i$ then $u$ is adjacent to $v_i$ and $v_{i+1}$. If $u=q_i$ then $u$ is adjacent to $v_i$ and $v_{i-1}$.
This proves $(f_{k+1}$).


\textit{Proof of ($g_{k+1}$)}:   Let $P_i^\prime = w_{i,1}\mbox{-}\cdots\mbox{-} w_{i,l_i}$ where $l_i = d-1$ or $d-2$, $w_{i,1} = p_i$ and $w_{i,l_i} = q_i$. Let $u\in P_i^\prime$.

As in the above proof of ($f_{k+1}$), if $u = w_{i,j}$ for $1<j<l_i$ then $\mathrm{lk}_{X_{k+1}}(u) = w_{i,j-1} \mbox{-} v_i \mbox{-} w_{i,j+1}$ and hence $\deg_{X_{k+1}}(u) = 3$.
If $u = p_i = w_{i,1}$ then  $$\mathrm{lk}_{X_{k+1}}(u) =  \mathrm{lk}_{X_{k+1}}(p_i)  = w_{i,2} \mbox{-} v_i \mbox{-} v_{i+1} \mbox{-}  w_{i+1,l_{i+1}-1}$$ and hence $\deg_{X_{k+1}}(u) = 4$.
Similarly, if $u=q_i$ then $\deg_{X_{k+1}}(u)=4$.
This proves ($g_{k+1}$). \end{proof}

\subsection{Vertex count}

The vertex set $V_k$ of the triangulated disc $X_k$ can be partitioned into three subsets as follows:

Let $A_k := \{ u \in V_k : u \text{ has two neighbours in } V_{k-1}\}$,

$B_k := \{ u \in V_k : u \text{ has a unique neighbour in } V_{k-1} \text{ which lies in }A_{k-1}\}$, and

$C_k := \{ u \in V_k : u \text{ has a unique neighbour in } V_{k-1} \text{ that is not in }A_{k-1}\}$.

Thus,
\begin{equation}\label{decomp}
V_k = A_k \sqcup B_k\sqcup C_k
\end{equation}
 for $k\geq 2$.

\begin{figure}
  \centering
  \includegraphics[scale=.45]{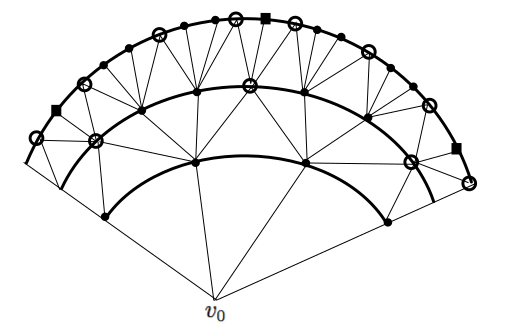}\\
  \caption{The boundary vertices of a layer $L_k$ in the triangulation are of three kinds $A_k$ (marked by a circle), $B_k$ (marked by a square), $C_k$ (marked by a disk). }
  \label{bdry}
\end{figure}

We obtain the following vertex count:\\

\begin{lem}\label{vcount}
Let $d>6$ and $k\geq 1$.  Then
$$\# V_k = \frac{d}{\sqrt{(d-6)(d-2)}} \left( \left(\frac{d-4 + \sqrt{(d-6)(d-2)}}{2}\right)^k - \left(\frac{d-4 - \sqrt{(d-6)(d-2)}}{2}\right)^k\right).$$

\end{lem}

\begin{proof}

Let $n_k = \#V_k$.

Clearly, $\#A_k =  \#\text{ edges in } L_{k-1} = \# V_{k-1} =  n_{k-1}$.

For each element in $A_{k-1}$, there are exactly $d-6$ elements in $B_k$, and hence $$\#B_k = \# A_{k-1} \times (d-6) = n_{k-2}(d-6).$$

By the same argument, $\#C_k = (n_{k-1}-n_{k-2})(d-5)$.

Thus, we have
 $$n_k = n_{k-1} + n_{k-2}(d-6) + (n_{k-1} - n_{k-2})(d-5) = (d-4)n_{k-1} - n_{k-2}$$ for $k\geq 3$.

Also, $n_0=1, n_1=d$ and $n_2 = n_1 + n_1(d-5) = d(d-4)$.

Let 
\begin{equation}\label{init}
y_0 = 0, y_1=d
\end{equation}

and $y_k = n_k$ for $k\geq 2$.

Then we have the difference equation
\begin{equation}\label{diffeq}
y_k - (d-4)y_{k-1} + y_{k-2} = 0
\end{equation}
for $k\geq 2$.

(Note that defining the new sequence by the shift allows us to start with a easier set of initial conditions (\ref{init}).)

We now employ the standard technique of solving this difference equation:

The characteristic equation is $$\lambda^2 - (d-4)\lambda + 1=0$$

such that the two roots are 
$$\lambda = \frac{d-4 \pm \sqrt{(d-6)(d-2)}}{2}. $$

For $d>6$,  these roots are distinct and we have
$$y_k = c_1\left(\frac{d-4 + \sqrt{(d-6)(d-2)}}{2}\right)^k  + c_2 \left(\frac{d-4 - \sqrt{(d-6)(d-2)}}{2}\right)^k$$

Using the initial conditions (\ref{init}) we obtain
$$c_1 = -c_2 = \frac{d}{\sqrt{(d-6)(d-2)}} $$

which completes the proof.
\end{proof}

\textit{Remark.}  The values of $n_k$ for small $k$ and $d=7$ and $8$, are shown in the table below:
\begin{center}
  \begin{tabular}{|l|l|l|l|l|l|} 
    \hline
    $n_1$ & $n_2$ & $n_3$ & $n_4$ & $n_5$ & $n_6$\\ \hline
    $k=7$ & 21 & 56  & 147  & 385 & 1008  \\ \hline
    $k=8$  & 32 & 120  & 448  & 1672 & 6240   \\
    \hline
  \end{tabular}
\end{center}

This vertex count is consistent with the exponential growth (in $k$)  of the circumference of a hyperbolic disk of radius $k$.
In contrast, for $d=6$ the roots of the characteristic equation for (\ref{diffeq}) are equal, and we obtain the linear growth $n_k =6k$ for $k\geq 1$.

\medskip

\subsection{Uniqueness}
Finally, we have:

\begin{lem}\label{lemma5}
$X_k$ determines $X_{k+1}$ uniquely for each $k\geq 1$.
\end{lem}

\begin{proof}

Assume $d\geq 7$ and let $k\geq 1$.

Let $L_k $ be the cycle $C_{n_k}(v_1, \ldots v_{n_k})$ where $n_k = \# V_k$ as in the previous lemma.

Assume without loss of generality that $v_1 \in C_k$, where $V_k = A_k \sqcup B_k \sqcup C_k$ as in the previous lemma.

Then we have $\mathrm{lk}_{X_k}(v_1) = v_2 \mbox{-}w_1\mbox{-}v_{n_k}$  where $w_1\in V_{k-1}$ as shown in Figure 7.

(We can also assume that $w_1 \in C_{k-1}$ for $k\geq 2$.) \\

\begin{figure}[h]
  \centering
  \includegraphics[scale=.45]{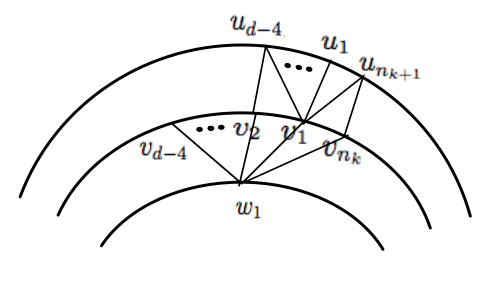}\\
  \caption{ }
\end{figure}

By Lemma \ref{lemma2} we know that $\mathrm{lk}_{X_{k+1}}(v_1) = \mathrm{lk}_X(v_1)$ is of the form $$C_d(v_2,w_1,v_{n_k}, u_{n_{k+1}}, u_1,\ldots ,u_{d-4})$$
where  $u_{n_{k+1}} , u_{d-4} \in A_{k+1}$ , $u_1,\ldots ,u_{d-5} \in C_{k+1}$ and the path $u_{n_{k+1}} \mbox{-} u_1\mbox{-}\cdots \mbox{-} u_{d-4}$ is in $L_{k+1}$.
Then $v_1u_{n_{k+1}}u_1$, $v_1u_1u_2$, $\ldots$, $v_1u_{d-5}u_{d-4}$, $v_1 v_{n_k}u_{n_{k+1}}$, $v_1v_2v_{d-4}$ are $2$-simplices in the triangulation.

We shall now repeat the same argument to determine the remaining triangles of $X_{k+1}$:

Assume we know $\mathrm{lk}_X(v_1),\ldots ,\mathrm{lk}_X(v_i)$ for $1\leq i\leq n_k-1$.

Note that we know $X_k$ and hence we know the nature of $v_i$, namely whether it lies in $A_k,B_k$ or $C_k$.\\

\textit{Case 1: $v_{i+1} \in C_k \sqcup B_k$. } (See Figure 8 (a).)

In this case, by the argument above, $\mathrm{lk}_X(v_{i+1})$ is uniquely determined to be  $C_d(v_{i+2}, w_j, v_i, u_l,\ldots ,u_{l+d-4})$ for some  vertices $u_{l+1}, \ldots u_{l+ d-4} \in V_{k+1}$ (Again by the same argument as above and in the proof of Lemma \ref{lemma2}$, u_l, u_{l+d-4} \in A_{k+1}$ and $u_{l+1}, \ldots ,u_{l+d-5} \in C_{k+1}$.)   Also, the path $u_l\mbox{-}u_{l+1}\mbox{-}\cdots \mbox{-} u_{l+d-4}$ is in $L_{k+1}$.

\textit{Case 2: $v_{i+1} \in A_k$. } (See Figure 8 (b).)

In this case, $\mathrm{lk}_{X_k}(v_{i+1})$  is of the form $v_{i+2}\mbox{-}w_{j+1}\mbox{-}w_j\mbox{-}v_i$. Since $v_{i+1}$ is an internal vertex of $X_{k+1}$, we have that $\mathrm{lk}_{X_{k+1}}(v_{i+1})$ is a $d$-cycle (and hence equals $\mathrm{lk}_{X}(v_{i+1})$). Assume it is $C_d(v_{i+2}, w_{j+1},w_j,v_i, u_l,u_{l+1}, \ldots u_{l+d-5})$ for some vertices $u_{l+1}, \ldots ,u_{l+d-5} \in V_{k+1}$ and $\{u_{l+1}, \ldots ,u_{l+d-5}\} \cap \{u_0,\ldots ,u_{l-1}\} = \emptyset$ (or $\{u_1\}$ if $i+1 = n_k$).

It follows that $u_l, u_{l+d-5} \in A_{k+1}$ ad $ u_{l+1},\ldots ,u_{l+d-6} \in B_{k+1}$.

By induction, this determines $\mathrm{lk}_X(v_{i})$ uniquely for  all $i=1,2,\ldots ,n_k$. Note that we also determine  which of the subsets of $A_k,B_k,C_k$ each vertex belongs to (see \eqref{decomp}) , and hence the triangles in $X_{k+1}$.

Therefore, by the definition of $X_j$ in (\ref{xk}) we get $X_{k+1}$ uniquely.
\end{proof}

\begin{figure}
  \centering
  \includegraphics[scale=.4]{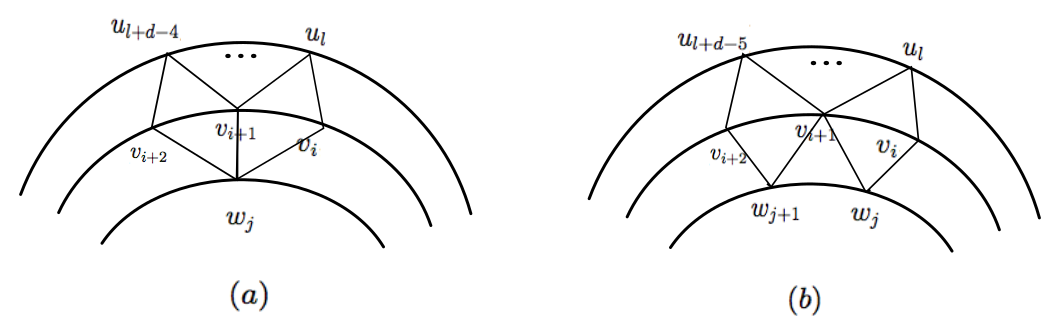}\\
  \caption{ }
\end{figure}

Existence has already been proved in the previous section. This completes the proof of Theorem \ref{thm1}.

\section{Proof of Theorem \ref{thm0}}

The proof of Theorem \ref{thm0} can be now easily deduced:

Let $S$ be a  $d$-regular triangulated surface. 

By Lemma \ref{lem1}, when $d=3,4,5$, the surface $S$ is a sphere or projective plane:

In the first case, the triangulation is combinatorially equivalent to the standard triangulation of the boundary of the tetrahedron, octahedron or icosahedron. Since the boundary of each such Platonic solid can be radially projected to a triangulation  of the sphere, and is hence geometric with respect to the spherical (constant positive curvature) metric. The remaining case of the projective plane is obtained as a two-fold quotient of the icosahedral triangulation; the quotient is by an antipodal map that is an isometry, and the resulting triangulation is also geometric.

When $d\geq 6$, the surface admits a constant curvature metric by Lemma \ref{lem2}.   The triangulation then lifts to a $d$-regular triangulation of the corresponding model space ($\mathbb{E}^2$ or $\mathbb{H}^2$) which is homeomorphic to $\mathbb{R}^2$. 

In the case that $d=6$, such a triangulation of $\mathbb{E}^2$ is combinatorially equivalent to the hexagonal triangulation by a result of \cite{Datta1}.
Applying Theorem \ref{thm1} to the case when $d>6$, the triangulation of $\mathbb{H}^2$ is combinatorially equivalent to the hexagonal tiling of $\mathbb{E}^2$ or the geodesic tiling of $\mathbb{H}^2$  generated by the Schwarz triangle group $\Delta(3,d)$. 

In either case, the original surface $S$ is obtained by isometric identifications along the edges of a (possibly infinite) fundamental domain, and acquires a geometric triangulation that is combinatorially equivalent to the original one.

\bibliographystyle{amsalpha}

\bibliography{Tiling-refs}

\end{document}